\documentclass[11pt,a4paper]{article}
\usepackage[utf8]{inputenc}
\usepackage[english]{babel}
\usepackage{amsmath}
\usepackage{amsfonts}
\usepackage{amssymb}
\usepackage{graphicx}
\usepackage{amsthm}
\usepackage{amssymb, amscd}
\usepackage[all, knot]{xy}
\usepackage{multicol}
\usepackage{varioref}

\allowdisplaybreaks[3]

\newtheorem{thm}{Theorem}[section]
\newtheorem{cor}[thm]{Corollary}
\newtheorem{lem}[thm]{Lemma}
\newtheorem{prop}[thm]{Proposition}

\theoremstyle{definition}
\newtheorem{defn}[thm]{Definition}
\theoremstyle{remark}
\newtheorem{rem}[thm]{Remark}
\numberwithin{equation}{section}

\newcommand{\C}{\mathbb{C}}
\newcommand{\Z}{\mathbb{Z}}

\newcommand{\K}{\mathbb{K}}

\author{Abdenacer  Makhlouf \thanks{Universit\'e de Haute Alsace, LMIA,
4, rue des Fr\`eres Lumi\`ere
F-68093 Mulhouse, France
(\textbf{abdenacer.makhlouf@uha.fr})} \and  Nejib Saadaoui  \thanks{ Higher Institute of Computer Science Medenine, Av. Djorf Medenine km 3, PO Box 283, 4100 Medenine, Tunisia  (\textbf{nejib.saadaoui@fsg.rnu.tn})}}
\title{ Cohomology and Deformation  of  Virasoro Extensions of $q$-Witt  Hom-Lie superalgebra }
\begin{document}
\maketitle
\begin{abstract}
The purpose of this paper is to study Virasoro extensions of the $q$-deformed Witt Hom-Lie superalgebra. Moreover, we  provide the cohomology and  deformations of the Ramond Hom-superalgebra and special Ramond Hom-superalgebra. 

\end{abstract}
\section{Introduction}
Various important examples of Lie superalgebras have been constructed
starting from the Witt  algebra $\mathcal{W}$. It is well-known that $ \mathcal{W}$ (up to equivalence and rescaling) has a unique nontrivial one-dimensional central extension, the Virasoro algebra. This is not the case in the superalgebras case, very important examples are the Neveu-Schwarz and the
Ramond  superalgebras. For further generalizations, we refer to Schlichenmaier's book \cite{SchlichenBook}. The Neveu-Schwarz and Ramond superalgebras are usually called super-Virasoro algebras since they can be viewed as super-analogs of the Virasoro algebra.
  Their corresponding second cohomology groups are computed in \cite{Hijli}.  One may found the second cohomology group computation  of Witt and Virasoro algebras  in \cite{Alice,FialowskiRigidityWitt, Schlichen}. The $q$-deformed Witt superalgebra $\mathcal{W}_q$  was  defined in \cite{AmmarMakhloufJA2010} as a main example of Hom-Lie superalgebras. The cohomology  and  deformations of  $\mathcal{W}_q$ were studied in \cite{Saadaoui,SaadaouiN}. 
The first and second cohomology groups of the $q$-deformed Heisenberg-Virasoro algebra of Hom-type  are computed in \cite{Cheng}. 

In this paper, we aim to study extensions of Hom-Lie superalgebras and discuss mainly the case of  $\mathcal{W}_q$ Hom-superalgebra. We provide a characterization of the Virasoro extensions of the $q$-Witt superalgebra and study their cohomology and deformations. In   Section 2,  we review the basics about Hom-Lie superalgebras and their cohomology; and in Section 3, we discuss their extensions.
In Section $4$, we describe the $q$-Witt superalgebra extensions of Virasoro type, we introduce  Ramond Hom-superalgebra and special  Ramond Hom-superalgebra. Section $5$ is dedicated to  cohomology and derivations calculations of Virasoro extensions of $q$-Witt superalgebra. In the last section we discuss one-parameter formal  deformations  of  Ramond and special  Ramond Hom-Lie superalgebras.
\section{Preliminaries}
In this section,  we recall   definitions of Hom-Lie superalgebras, $q$-deformed Witt superalgebra and some basics about representations and cohomology. For more details we refer to 
 \cite{Saadaoui}.

\begin{defn}
A Hom-Lie superalgebra  is a triple $(\mathcal{G},\ [.,.],\ \alpha)$\ consisting of a superspace $\mathcal{G}$, an even bilinear map \ $\ [.,.]:\mathcal{G}\times \mathcal{G}\rightarrow \mathcal{G}$ \ and an even superspace homomorphism \ $ \alpha:\mathcal{G}\rightarrow \mathcal{G} \ $satisfying
\begin{eqnarray*}
&&[x,y]=-(-1)^{|x||y|}[y,x],\\
&&(-1)^{|x||z|}[\alpha(x),[y,z]]+(-1)^{|z||y|} [\alpha(z),[x,y]]+(-1)^{|y||x|} [\alpha(y),[z,x]]=0,\label{jacobie}
\end{eqnarray*}
for all homogeneous element $x, y, z$ in $\mathcal{G}$ and where $|x|$ denotes the degree of the homogeneous element $x$.
\end{defn}
 \subsection{A $q$-deformed Witt superalgebra}\label{witt}
A $q$-deformed Witt superalgebra  $\mathcal{W}^q $ can  be presented as the $\Z_2$-graded vector space
with $\{L_n\}_{n\in\Z}$ as a basis of the even homogeneous part and $\{G_n\}_{n\in\Z}$
 as a basis of the
odd homogeneous part. It is equipped with the commutator
\begin{align}&[L_{n},L_{m}]=(\{m\}-\{n\})L_{n+m}, \label{crochet1} \\
&[L_{n},G_{m}]=(\{m+1\}-\{n\})G_{n+m},\label{crochet2}
\end{align}
where $\{m\}$ denotes the $q$-number $m$, that is $\{m\}=\frac{1-q^m}{1-q}$.
 The other brackets are obtained by supersymmetry or are equal to $0$.  The even linear map
$\alpha$  on $\mathcal{W}^{q}$ is  defined on the generators by
    $$ \alpha(L_{n})=(1+q^{n})L_{n}, \ 
\alpha(G_{n})=(1+q^{n+1})G_{n}.$$ For more details, we refer to \cite{AmmarMakhloufJA2010}.
\subsection{Cohomology of Hom-Lie superalgebras}
Let $(\mathcal{G},[.,.],\alpha)$ be a Hom-Lie superalgebra and $V=V_{0}\oplus V_{1}$ be  an arbitrary vector superspace. Let $\beta\in\mathcal{G}l(V)$ be an arbitrary even linear self-map on $V$  and
\begin{eqnarray*}
[.,.]_{V}&:&\mathcal{G}\times V \rightarrow V \\
&&(g,v)\mapsto [g,v]_{V}
\end{eqnarray*}
 a bilinear map  satisfying $[\mathcal{G}_{i},V_{j}]_{V}\subset V_{i+j}$ where $i,j\in \mathbb{Z}_{2}.$
\begin{defn}
The triple $(V,[.,.]_{V}, \beta)$  is called a representation of the Hom-Lie superalgebra $\mathcal{G}=\mathcal{G}_{0}\oplus \mathcal{G}_{1}$ or  $\mathcal{G}$-module $V$  if the  even bilinear map $[.,.]_{V}$ satisfies, for $x,y\in \mathcal{G}$ and $v\in V$, 
\begin{eqnarray}
\left[[x,y],\beta(v)\right]_{V}&=&\left[\alpha(x),[y,v]\right]_{V}-(-1)^{|x||y|}\left[\alpha(y),[x,v]\right]_{V} .\label{rep2}
\label{mod}
\end{eqnarray}
\end{defn}

\begin{rem}
When $[.,.]_{V}$ is the zero-map, we say that the module $V$ is trivial.
\end{rem}
\begin{defn}\cite{Saadaoui}
The set $C^{k}(\mathcal{G},V)$ of $k$-cochains on space $\mathcal{G}$ with values in $V,$ is the set of $k$-linear maps $f:\otimes^{k}\mathcal{G}\rightarrow V$
satisfying
$$f(x_{1},\dots,x_{i},x_{i+1},\dots,x_{k})=-(-1)^{|x_{i}||x_{i+1}|}f(x_{1},\dots,x_{i+1},x_{i},\dots,x_{k})\ \textrm{ for } 1\leq i\leq k-1 .$$
For $k=0$ we have $C^{0}(\mathcal{G},V)=V.$\\
Define $\delta^{k}:C^{k}(\mathcal{G},\ V)\rightarrow C^{k+1}(\mathcal{G},\ V)$ by setting
\begin{align}
&\delta^{k}(f)(x_{0},\dots,x_{k})=\nonumber \\&\sum_{0\leq s < t\leq k}(-1)^{t+|x_{t}|(|x_{s+1}|+\dots+|x_{t-1}|)}
f\Big(\alpha(x_{0}),\dots,\alpha(x_{s-1}),[x_{s},x_{t}],\alpha(x_{s+1}),\dots,\widehat{x_{t}},\dots,\alpha(x_{k})\Big) \nonumber \\
&+\sum_{s=0}^{k}(-1)^{s+|x_{s}|(|f|+|x_{0}|+\dots+|x_{s-1}|)}\Bigg[\alpha^{k+r-1}(x_{s}), f\Big(x_{0},\dots,\widehat{x_{s}},\dots,x_{k}\Big)\Bigg]_{V},\label{def cobbb}
\end{align}
 where $f\in C^{k}(\mathcal{G},\ V)$,  $|f|$ is the parity of $f$, $\ x_{0},....,x_{k}\in \mathcal{G}$ and $\ \widehat{x_{i}}\ $  means that $x_{i}$ is omitted.\\
 \end{defn}
We assume that the representation $(V,[.,.]_{V},\beta)$ of a Hom-Lie superalgebra $(\mathcal{G},[.,.],\alpha)$ is trivial.
Since $ [.,.] _V = $ 0, the operator defined in \eqref{def cobbb} becomes

\begin{align}\label{Def cob}
& \delta^{k}_T(f)(x_{0},\dots,x_{k}) =\\& \sum_{0\leq s < t\leq k}(-1)^{t+|x_{t}|(|x_{s+1}|+\dots+|x_{t-1}|)} 
f\Big(\alpha(x_{0}),\dots,\alpha(x_{s-1}),[x_{s},x_{t}],\alpha(x_{s+1}),\dots,\widehat{x_{t}},\dots,\alpha(x_{k})\Big). \nonumber
\end{align}


The pair $(\oplus_{k>0}C^{k}_{\alpha ,\beta}(\mathcal{G},\ V),\{\delta^{k}\}_{k>0})$ defines a chomology complex, that is
 $\delta^{k} \circ \delta^{k-1}=0.$
\begin{itemize}
\item The $k$-cocycles space is defined as $Z^{k}(\mathcal{G})=\ker \ \delta^{k}.$
 \item  The $k$-coboundary space is defined as    $B^{k}(\mathcal{G})=Im \ \delta^{k-1}.$
  \item The $k^{th}$ cohomology  space is the quotient  $H^{k}(\mathcal{G})=Z^{k}(\mathcal{G})/ B^{k}(\mathcal{G}). $ It decomposes as well as even and odd $k^{th}$ cohomology spaces.
 \end{itemize}

Now,  we consider the adjoint representation of a Hom-superalgebra and define the first and second coboundary maps. For all $f\in C^{1}_{\alpha,\alpha}(\mathcal{G},\mathcal{G})=\{g\in C^{1}(\mathcal{G},\mathcal{G}); g\circ \alpha =\alpha \circ g\}$
 the operator defined in \eqref{def cobbb} $(r=0, \ k\in \{1,2\})$ becomes
\begin{eqnarray}\label{adjoint1}
\delta^{1}_{\mathcal{G}}(f)(x,y)&=&-f([x,y])+(-1)^{|x||f|}[x,f(y)]-(-1)^{|y|(|f|+|x|}[y,f(x)]
\end{eqnarray}
\begin{eqnarray}\label{adjoint2}
\delta^{2}_{\mathcal{G}}(f)(x,y,z)&=&-f([x,y],\alpha(z))+(-1)^{|z||y|}f([x,z],\alpha (y))+f(\alpha (x),[y,z])\nonumber\\
&&
+(-1)^{|x||f|}[\alpha(x),f(y,z)]
-(-1)^{|y|(|f|+|x|}[\alpha(y),f(x,z)]\\
&&+(-1)^{|z|(|f|+|x|+|y|}[\alpha(z),f(x,y)]\nonumber
\end{eqnarray}
Then we have
$$\delta^{2}_{\mathcal{G}}\circ \delta^{1}_{\mathcal{G}}(f)=0, \qquad \forall f\in C^{1}_{\alpha,\alpha}(\mathcal{G},\mathcal{G}).$$
We denote by $H^{1}(\mathcal{G},\mathcal{G})$ (resp. $H^{2}(\mathcal{G},\mathcal{G})$) the corresponding 1st and 2nd cohomology groups.
An element $f$ of $Z^1(\mathcal{G},\mathcal{G})$ is called derivation of $\mathcal{G}$.
\section{Extensions of   Hom-Lie superalgebras}\label{extensions}
An extension of a Hom-Lie superalgebra $(\mathcal{G},[.,.],\alpha)$  by a  representation $(V,[.,.]_V,\beta )$   is an exact sequence
$$0\longrightarrow (V,\beta)\stackrel{i}{\longrightarrow} (K,\gamma)\stackrel{\pi}{\longrightarrow }(\mathcal{G},\alpha) \longrightarrow 0 $$
 satisfying $\gamma\ o\  i=i\ o\ \beta $ and $\alpha \ o \ \pi = \pi\  o \ \gamma.$\\
   This extension is said to be   central if $[K,i(V)]_{K}=0.$
   
 In particular, if $K=\mathcal{G}\oplus V$, $i(v)=v, \ \forall v\in V$ and $\pi(x)=x, \ \forall x\in \mathcal{G}$, then we have $\gamma(x,v)=(\alpha(x),\beta(v))$ and we denote
 $$0\longrightarrow (V,\beta)\longrightarrow (K,\gamma)\longrightarrow (\mathcal{G},\alpha) \longrightarrow 0 .$$
For convenience, we introduce the following notation for certain cochains  spaces on $K = \mathcal{G}\oplus V$, 
 $\mathcal{C}(\mathcal{G}^n,\mathcal{G})$ and 
$ \mathcal{C}(\mathcal{G}^kV^l,V),$ 
where $\mathcal{G}^kV^l$  is the subspace of $K^{k+l}$ determined by products of $k$
elements from $\mathcal{G}$ and $l$ elements from $V$.

Let $(\phi,\psi)\in \mathcal{C}^2(K,K)\times \mathcal{C}^2(K,K)$. We set, for all $X,Y,Z\in K_{0}\cup K_{1}$,
\begin{align*}
&\phi\circ\psi(X,Y,Z)=(-1)^{|X||Z|}\circlearrowleft_{X,Y,Z}(-1)^{|X|(|\psi|+|Z|)}\phi(\gamma(X),\psi(Y,Z)), 
\end{align*}
and
\begin{align*}
[\phi,\psi]&=\phi\circ\psi+(-1)^{|\psi||\phi|}\psi\circ\phi.\\
\end{align*}
For $f\in \mathcal{C}^{2}(K,K)$, we set  $f=\widetilde{f}+\widehat{f}+\overline{f}+v+\widehat{v}+\overline{v}$ where $\widetilde{f}\in \mathcal{C}( \mathcal{G}^{2}, \mathcal{G})$, 
 $\widehat{f}\in \mathcal{C}( \mathcal{G} V, \mathcal{G})$,  $\overline{f}\in \mathcal{C}( V^{2}, \mathcal{G})$ ,  $v\in \mathcal{C}(\mathcal{G}^2,V)$, ,  $\widehat{v}\in \mathcal{C}( \mathcal{G} V,V )$ and  $\overline{v}\in \mathcal{C}(  V^2,V )$. \\ Let $d$ $\in \mathcal{C}^2(K,K),$
if $(K,d,\gamma)$ is a Hom-Lie superalgebra and $V$ is an ideal in $K$ (i.e $d(\mathcal{G},V)\subset V$ ), we obtain by using the above notation :
\begin{center}
\begin{itemize}
 \item $\widehat{d}\equiv0$,
   \item $\overline{d}\equiv0$,
 \item $0=[d,d](x,y,z)=\Big([\widetilde{d},\widetilde{d}]+2  [v,\widetilde{d}]+2  [\widehat{v},v]\Big)(x,y,z)$,\qquad $\forall x,y,z\in \mathcal{G}$,\\
 $[\widetilde{d},\widetilde{d}](x,y,z)\in \mathcal{G}$ and  $\Big(2  [v,\widetilde{d}]+2  [\widehat{v},v]\Big)(x,y,z)\in V,$
 \item $0=[d,d](x,y,w)=\Big([\widehat{v},\widehat{v}]+2[\widetilde{d},\widehat{v}]+2[\overline{v},v]\Big)(x,y,w)$,\quad $\forall (x,y,w)\in \mathcal{G}^2\times V,$
  \item   $0=[d,d](x,v,w)=2[\overline{v},\widehat{v}](x,v,w)$,\qquad $\forall (x,v,w)\in \mathcal{G}\times V^2,$
 \item $0=[d,d](u,v,w)=\frac{1}{2}[\overline{v},\overline{v}](u,v,w)$,\qquad $\forall (u,v,w)\in  V^3$.
\end{itemize}
\end{center}
We deduce the following theorem
\begin{thm}
The triple $(K,d,\gamma)$ is a Hom-Lie superalgebra if and only  if the following conditions are satisfied
\begin{center}
\begin{itemize}
\item $(\mathcal{G},\widetilde{d},\alpha)$ is a Hom-Lie superalgebra,
\item  $[v,\widetilde{d}]+  [\widehat{v},v]\equiv0$,
  \item  $\frac{1}{2}[\widehat{v},\widehat{v}]+[\widetilde{d},\widehat{v}]+[\overline{v},v]\equiv0$,
    \item   $[\overline{v},\widehat{v}]\equiv0$,
    \item $(V,\overline{v},\beta)$ is a Hom-Lie superalgebra.
\end{itemize}
\end{center}
\end{thm}
\begin{cor}
If $\overline{v}\equiv0$, then 
the triple $(K,d,\gamma)$ is a Hom-Lie superalgebra if and only  if the following conditions are satisfied
\begin{itemize}
\item $(\mathcal{G},\widetilde{d},\alpha)$ is a Hom-Lie superalgebra,
\item  $(V,\widehat{v},\beta)$ is a representation of $\mathcal{G}$,
\item  $v$ is a $2$-cocycle on  $V$ (with the cohomology defined by  $(\mathcal{G},\widetilde{d},\alpha)$
and $(V,\widehat{v},\beta)$).
\end{itemize}
\end{cor}
\begin{cor}
Let $$0\longrightarrow (V,\beta){\longrightarrow} (\mathcal{G}\oplus V,\widetilde{\alpha}){\longrightarrow }(\mathcal{G},\alpha) \longrightarrow 0 $$
 be an extension of $(\mathcal{G},[., .], \alpha)$ by   a  representation $(V,[.,.]_V,\beta )$, where $\widetilde{\alpha}$ is defined by $\widetilde{\alpha}(x,v)=(\alpha(x),\beta(v)$, for all $x\in \mathcal{G}$ and $v\in V$. 
 
 Let $\varphi \in (C^{2}(\mathcal{G},V))_{j},\ (j\in \Z_2)$. We define a skew-symmetric bilinear bracket operation $d:\wedge^2(\mathcal{G}\oplus V)\rightarrow \mathcal{G}\oplus V $ by
\begin{equation}\label{crochetViraP4}
  d\Big((x,u);(y,v)\Big)=\Big([x,y],[x,v]_{V}-(-1)^{|u||y|}[y,u]_{V}+\varphi(x,y)\Big)\ \ \ \forall x,y \in \mathcal{G} \     v, w\in V.
 \end{equation}
The triple  $(\mathcal{G}\oplus V,d,\widetilde{\alpha})$   is a Hom-Lie superalgebra if and only if $\varphi$ is a  $2$-cocycle (i.e. $\varphi\in Z^2(\mathcal{G},V)$).
\end{cor}
\begin{proof}
We have 
$\overline{v}\equiv0$, $\widetilde{d}=[.;.] $,
 $v=\varphi $,
 $\widehat{v}=[.,.]_{V}$ and 
 $\overline{v}\equiv0$. Then, we deduce
\begin{itemize}
\item $(\mathcal{G},\widetilde{d},\alpha)$ is a Hom-Lie superalgebra
\item  $(V,\widehat{v},\beta)$ is a representation of $\mathcal{G}$.
\end{itemize}
Therefore, the triple  $(\mathcal{G}\oplus V,d,\gamma)$   is a Hom-Lie superalgebra if and only if $\varphi$ is a  $2$-cocycle (i.e. $f\in Z^2(\mathcal{G},V)$).
\end{proof}

\begin{rem}
\begin{itemize}
 \item If $\varphi$ is even, then $\mathcal{G}_0\oplus V_0$ is an even homogeneous part and $\mathcal{G}_1\oplus V_1$ is the odd
 homogeneous part of $\mathcal{G}\oplus V$.
  \item  If $\varphi$ is odd, then, $\mathcal{G}_0\oplus V_1$ is an even homogeneous part and $\mathcal{G}_1\oplus V_0$ is the odd
 homogeneous part of $\mathcal{G}\oplus V$.  The Hom-Lie superalgebra  $\mathcal{G}\oplus V$ is called the special extension of $\mathcal{G}$ by $V$.
\end{itemize}
\end{rem}
\begin{thm}
Let $(V,[.,.]_V,\beta)$  be a representation of  a Hom-Lie superalgebra $(\mathcal{G}, [. , . ], \alpha)$.
 The even second cohomology space $H_{0}^2(\mathcal{G}, V ) = Z_0^2(\mathcal{G}, V )/B_0^2(\mathcal{G}, V )$
is in one-to-one correspondence with the set of the equivalence classes
extensions of $(\mathcal{G}, [., .], \alpha)$ by $(V,[.,.]_V,\beta)$ .
\end{thm}
\begin{proof}
Let $(\mathcal{G}\oplus V,d,\gamma)$ and $(\mathcal{G}\oplus V,d',\gamma)$  be  two extensions of $(\mathcal{G}, [., .], \alpha)$. So there are two even cocycles $\varphi$ and $\varphi'$ such as $ d\Big((x,u);(y,v)\Big)=\Big([x,y],[x,v]_{V}+[u,y]_{V}+\varphi(x,y)\Big)$ and $ d'\Big((x,u);(y,v)\Big)=\Big([x,y],[x,v]_{V}+[u,y]_{V}+\varphi'(x,y)\Big)$.\\
If  $ \varphi-\varphi'=\delta^1 h(x,y)$, where $h:\mathcal{G}\rightarrow V$ is a linear map satisfying $h\circ\alpha=\beta    \circ h$ (i.e $\varphi-\varphi'\in B^2(\mathcal{G},V)$    ). Let us define $\Phi :(\mathcal{G}\oplus V,d,\gamma) \rightarrow           (\mathcal{G}\oplus V,d',\gamma)$ by
$\Phi(x,v) =(x,v-h(x))$. It is clear that $\Phi$ is bijective. Let us check that $\Phi$ is a   Hom-Lie superalgebras homomorphism. We have
\begin{eqnarray*}
 & &d(\Phi((x,v)),\Phi(((y,w)))\\
 &=&d((x,v-h(x)),(y,w-h(y)))\\
  &=&([x,y],  [x,w]_{V}+[v,y]_{V}-[x,h(y)]_{V}-[h(x),y]_{V}+\varphi(x,y))                          \\
    &=&([x,y],  [x,w]_{V}+[v,y]_{V}-\delta^1(h)(x,y)+\varphi(x,y)-h([x,y]))                          \\
  &=&\Phi(([x,y],f(x,y)))\\
      &=&([x,y],  [x,w]_{V}+[v,y]_{V}+\varphi'(x,y)-h([x,y]) )                         \\
         &=&\Phi([x,y],  [x,w]_{V}+[v,y]_{V}+\varphi'(x,y) )                         \\
  &=&\Phi(d'((x,v),(y,w)).
\end{eqnarray*}

\end{proof}

\begin{defn}
\begin{itemize}
 \item
If $\varphi\equiv0$, the  Hom-Lie superalgebra  $(\mathcal{G}\oplus V,d,\gamma)$ is called the semidirect
product of the    Hom-Lie superalgebra  $(\mathcal{G},[.,.],\alpha)$    and $(V,[.,.]_{V},\beta)$.
 \item If $V=c\C$ ($c\in \C$) and $[., .]_{V}\equiv0$,  the  Hom-Lie superalgebra  $(\mathcal{G}\oplus V,d,\gamma)$ is
called a
one-dimensional central extension of  $\mathcal{G}$.
 \item If $V=c\C$, $[., .]_{V}\equiv0$  and $\mathcal{G}=\mathcal{W}^q$,  the  Hom-Lie superalgebra  $(\mathcal{G}\oplus V,d,\gamma)$ is
called a Virasoro Hom-superalgebra.
 \item If $V=c\C$, $[.,.]_{V}\equiv0$, $\varphi\equiv0$  and $\mathcal{G}=\mathcal{W}^q$,  the  Hom-Lie superalgebra  $(\mathcal{G}\oplus V,d,\gamma)$ is
called a trivial Virasoro Hom-superalgebra.\\
\end{itemize}
\end{defn}
\begin{defn}
A  Hom-Lie superalgebra  $(\mathcal{G},[.,.],\alpha)$ is said to be $\mathbb{Z}$-graded if
   $\displaystyle\mathcal{G}=\oplus_{n\in\mathbb{Z}}\mathcal{G}_{n}$, 
 where $dim (\mathcal{G}_{n})<\infty$,
  $\alpha(\mathcal{G}_{n})\subset \mathcal{G}_{n}$ and $[\mathcal{G}_{n},\mathcal{G}_{m}]\subset \mathcal{G}_{n+m}$,
 for all $n,\  m \in \mathbb{Z}$. For an element $x\in \mathcal{G}$, we call $n$ the degree of $x$, denoted $deg(x)=n$,
if $x\in \mathcal{G}_{n}$.
\end{defn}
\begin{prop}
Let $(\mathcal{G},[.,.],\alpha)$ be  a $\mathbb{Z}$-graded  Hom-Lie superalgebra.
We denote  $((\mathcal{G}\oplus c\C)_{0}=\mathcal{G}_{0}\oplus \C$ and $(\mathcal{G}\oplus c\C)_{n}=\mathcal{G}_{n}, \ \forall n\in\Z^{*}$. If $\varphi(\mathcal{G}_{n},\mathcal{G}_{m})\subset\delta_{n+m,0}\mathcal{G}_{n+m}$ then 
 $(\mathcal{G}\oplus c\C,d,\gamma)$ is also $\mathbb{Z}$-graded.\\
\end{prop}
\section{Virasoro Hom-superalgebras}
In this section, we describe  extensions of $q$-deformed Witt superalgebra $\mathcal{W}^{q}$. Let us recall the following result describing its  scalar cohomology:
\begin{thm}\cite{Saadaoui}\label{saadaoui}
$$H^{2}(\mathcal{W}^{q},\C)=\mathbb{C}[\varphi_{0}]\oplus\mathbb{C}[\varphi_{1}],
$$ where
\begin{eqnarray*}
\varphi_{0}(xL_{n}+yG_{m},zL_{p}+tG_{k})&=&xzb_{n}\delta_{n+p,0},\\
\varphi_{1}(xL_{n}+yG_{m},zL_{p}+tG_{k})&=&xtb_{n}\delta_{n+k,-1}-yz b_{p}\delta_{p+m,-1},\\
\end{eqnarray*}
with

\begin{equation*}
b_n=
\left\{ \begin{array}{ll}
\frac{1}{q^{n-2}} \frac{1+q^{2}}{1+q^{n}}\frac{\{n+1\}\{n\}\{n-1\}}{\{3\}\{2\}}          ,\quad \textrm{\ if $n\geq 0$,}\\
-b_{-n}\  \ \ \ \ \ \ \ \ \ \ \ \ \ \ \ \ \ \ \ \ \  \ \ \  \ \ \ \ \  \qquad  \textrm{\ if $n< 0$.}          \\
\end{array} \right.
\end{equation*}
\end{thm}
Using the even non trivial $2$-cocycle $\varphi_{0}$ defined in Theorem \ref{saadaoui}, we can define the followings Virasoro Hom-superalgebras :
\begin{itemize}
 \item
 The \emph{Neveu-Schwarz Hom-superalgebra} can be presented as the $\Z_2$-graded vector space with $\{L_n, D\}_{n\in \Z}$  as a basis of the even homogeneous part and $\{F_n\}_{n\in \frac{1}{2}+\Z}$  as a basis of the odd homogeneous part. It is equipped with the commutator

\begin{small}
\begin{align*}
&[L_n,L_m]=(\{m\}-\{n\})L_{n+m}+D \frac{1}{q^{n-2}} \frac{1+q^{2}}{1+q^{n}}\frac{\{n+1\}\{n\}\{n-1\}}{\{12\}}\delta_{n+m,0},&\\
 & \qquad\qquad\qquad\qquad\qquad\qquad \qquad\qquad \forall (n,m)\in(\Z_+\times\Z)\cup(\Z_-\times\Z_-)&\\
&[L_n,F_m]=(\{m+\frac{1}{2}\}-\{n\})F_{n+m}&\\
&[F_n,F_m]=[F_n,D]=[L_n,D]=0& .
\end{align*}
\end{small}
 \item
 The  \emph{Ramond Hom-superalgebra} satisfies the following commutation relations:
 \begin{small}
\begin{align*}
&[L_n+z,L_m+z']=(\{m\}-\{n\})L_{n+m}+c   \frac{1}{q^{n-2}} \frac{1+q^{2}}{1+q^{n}}\frac{\{n+1\}\{n\}\{n-1\}}{\{12\}}           \delta_{n+m,0}, \forall n>0\\
&[L_n+z,G_m+z']=(\{m+1\}-\{n\})G_{n+m}\\
&[G_n+z,G_m+z']=0.
\end{align*}
\end{small}

\end{itemize}

Using the odd non trivial $2$-cocycle $\varphi_{1}$ defined in Theorem \ref{saadaoui}, we can define the  \emph{special Ramond Hom-superalgebra}.
Then  it is equiped with the commutator
\begin{small}
\begin{align*}
& [L_n+z,L_m+z']=(\{m\}-\{n\})L_{n+m}& & \\
&[L_n+z,G_m+z']=(\{m+1\}-\{n\})G_{n+m}  +c   \frac{1}{q^{n-2}} \frac{1+q^{2}}{1+q^{n}}\frac{\{n+1\}\{n\}\{n-1\}}{\{12\}}           \delta_{n+m,0},\ \forall n>0\\
&[G_n+z,G_m+z']=0.
\end{align*}
\end{small}

In the above  cases, the map $\gamma$ is defined by 
\[\gamma(x,z)=(\alpha(x), z),\qquad       \forall x\in \mathcal{W}^{q}, \forall z\in\C,\]
  where $\alpha$ is given in Section \ref{witt}.
\begin{rem} We denote the  Ramond Hom-Lie superalgebra by $HR$, the  Neveu-Schwarz  Hom-Lie superalgebra by $HN$ and the special Ramond Hom-superalgebra by $SHR$.
The map $f:HR\rightarrow HN$ defined by $f(L_n)=L_n$, $f(G_n)=F_{n+\frac{1}{2}}$ and $f(c)=D$ is a Hom-Lie superalgebras isomorphism. 

Hence, Virasoro Hom-superalgebras are characterized in the following Theorem.
\end{rem}
\begin{thm}
Every Virasoro Hom-superalgebra is isomorphic to one of the following Virasoro Hom-Lie  superalgebras :
\begin{itemize}
 \item Ramond Hom-superalgebra,
    \item        Special Ramond Hom-superalgebra,
  \item      Trivial    Virasoro Hom-superalgebra.
 \end{itemize}
\end{thm}
 \section{Cohomology of Extensions of   Hom-Lie superalgebras  and Virasoro Hom-superalgebras}
Let $$0\longrightarrow (V,\beta){\longrightarrow} (K,d,\gamma){\longrightarrow }(\mathcal{G},\alpha) \longrightarrow 0 $$
 be an extension of $(\mathcal{G},\delta, \alpha)$ by   a  representation $(V,\lambda,\beta )$,  where $K=\mathcal{G}\oplus V$ and $d=\delta+\lambda+\varphi$  $(\varphi\in Z^2(\mathcal{G},V))$.
\subsection{ Derivations of Extensions of   Hom-Lie superalgebras}
If $f\in \mathcal{C}^{1}(K,K)$,  we set $f=\widetilde{f}+\widehat{f}+v+\widehat{v}$ where $\widetilde{f}\in \mathcal{C}^{1}(\mathcal{G},\mathcal{G})$,  $\widehat{f}\in \mathcal{C}^{1}(V,\mathcal{G})$, $v\in \mathcal{C}^{1}(\mathcal{G},V)$ and $\widehat{v}\in \mathcal{C}^{1}(V,V)$.
If $(\phi,\psi)\in \mathcal{C}^2(K,K)\times \mathcal{C}^1(K,K)$, we define
\begin{align*}
&\phi\circ\psi(X,Y)=\phi(\psi(X),\gamma^r(Y),)-(-1)^{|X||Y|}\phi(\psi(Y),\gamma^r(X),) \ \forall X,Y\in K_{0}\cup K_{1}\\
\end{align*}
and
\begin{align*}
[\phi,\psi]&=\phi\circ\psi-(-1)^{|\psi||\phi|}\psi\circ\phi.\\
\end{align*}
For $f\in \mathcal{C}^{1}(K,K)$, we have
\begin{eqnarray*}
 && \delta_{K}^{1}(f)((x+u,y+v))
  =\\&& -f(d(x+u,y+v))+d(f(x+u),\gamma^r(y+v))+(-1)^{|f||x+u|}d(\gamma^r(x+u),f(y+v)), \end{eqnarray*} which implies $[d,f]=\delta_{K}^{1}(f).$
 
  Then   \begin{equation}
\big(f\  \textrm{ is a $\alpha^r$- derivation of }\  K\big) \Leftrightarrow \big([d,f]\equiv0\big).
\end{equation}
 For all $x,y\in\mathcal{G},\ u,v\in V,$ we have
\begin{align*}
[d,f](x,y)&=\Big([\delta,\widetilde{f}]+[\delta+\lambda,v]+[\varphi,\widehat{f}+\widehat{v}]\Big)(x,y),\\
[d,f](x,v)&=\Big([\delta,\widehat{f}]+[\lambda,\widetilde{f}+\widehat{f}+\widehat{v}]+[\varphi,\widehat{f}]\Big)(x,v),\\
[d,f](u,v)&=0.
\end{align*}

Then, we deduce the following result :
\begin{thm}\label{f5derivationS}
 For all $x,y\in\mathcal{G},\ v\in V$, we have
\begin{equation}\label{5derivationS}
\big(f\ \textrm{is a }\alpha^r \text{-derivation of }\  K \big)\Leftrightarrow
\left\{ \begin{array}{ll}
\Big([\delta,\widetilde{f}]+[\varphi,\widehat{f}]\Big)(x,y)=0,\\
\Big(\big[\varphi,\widetilde{f}+\widehat{v}\big]+\big[\delta+\lambda,v\big]\Big)(x,y)=0,\\
\big[  \delta+\lambda,\widehat{f}\big](x,v)=0,\\
\Big( [\lambda,\widetilde{f}+\widehat{v}]+[\varphi,\widehat{f}]\Big)(x,v)=0.\\
\end{array} \right.
\end{equation}
\end{thm}

\subsection{ Derivations of Virasoro Hom-superalgebras}
Let recall the following result (for the proof see \cite{Saadaoui}).
\begin{lem}\label{P4lemma}
The set of $\alpha^{0}$-derivations of the Hom-Lie superalgebra $\mathcal{W}^{q}$ is $$Der_{\alpha^{0}}(\mathcal{W}^{q})=\langle D_{1}\rangle\oplus\langle D_{2}\rangle\oplus \langle D_{3}\rangle\oplus\langle D_{4}\rangle$$  where $D_{1}$, $D_{1}$, $D_{2}$, $D_{3}$ and $D_4$ are defined, with respect to the basis as
\begin{align*}
&D_{1}(L_{n})=nL_n,& \ &D_{1}(G_n)=nG_n,&\\
&D_{2}(L_{n})=0,& \ &D_{2}(G_n)=G_n,&\\
&D_{3}(L_{n})=nG_{n-1},&\ &D_{3}(G_n)=0,&\\
&D_{4}(L_{n})=0,&\ &D_{4}(G_n)=L_{n+1}.&
\end{align*}
\end{lem}

\

Let $(\mathcal{W}^{q}_{\varphi},d,\gamma)$ 
be a Hom-Virasoro-superalgebra. Then 
\[\mathcal{W}^{q}_{\varphi}=\mathcal{W}^{q}\oplus \C,\, d=[.,.]+\varphi,\, \textit{and } \gamma(x,z)=(\alpha(x),z),\]
   where $\varphi\in \C\varphi_{0}\cup\C\varphi_{1}$. 
\begin{lem}\label{P4lemma1}
\begin{equation*}
\big(f\ \textrm{ is a }\alpha^r \text{-derivation of } \mathcal{W}^{q}_{\varphi}\big) \Rightarrow \big(\widehat{f}\equiv0\big).\\
\end{equation*}
\end{lem}
\begin{proof}
Let  $f$ be an $\alpha^r$-derivation of  $\mathcal{W}^{q}_{\varphi}$.\\
Using Theorem \ref{f5derivationS} and $\lambda\equiv 0$,  $\forall z\in\C, \forall n\in\Z$ we have
\begin{align*}
\big[  \delta,\widehat{f}\big](L_n,z)&=0\\
&\Rightarrow -(-1)^{|x||z|}\delta(\widehat{f}(z),\alpha(L_n))=0,\\
&\Rightarrow[\widehat{f}(z),L_{n}]=0, \\
&\Rightarrow \widehat{f}(z)=0.
\end{align*}
\end{proof}
In the following, we provide the $\alpha^{0}$-derivations of  $\mathcal{W}^{q}_{\varphi}$ explicitly.
\begin{prop}
The set of $\alpha^{0}$-derivations of the Virasoro   Hom-Lie superalgebra  $\mathcal{W}^{q}_{\varphi_i}$ with $i=0,1$,  is $$Der_{\alpha^{0}}(\mathcal{W}^{q}_{\varphi_{i}})=\langle D_{1}\rangle\oplus\langle D_{2}\rangle\oplus \langle D_{3}\rangle\oplus\langle D_{4}\rangle,$$  where $D_{1}$, $D_{1}$, $D_{2}$, $D_{3}$ and $D_4$ are defined, with respect to the basis, as
\begin{align*}
&D_{1}(L_{n})=nL_n,&  &D_{1}(G_n)=nG_n,&D_{1}(1)=0,\\
&D_{2}(L_{n})=0,&  &D_{2}(G_n)=G_n,&D_{2}(1)=0,\\
&D_{3}(L_{n})=nG_{n-1},& &D_{3}(G_n)=0,&D_{3}(1)=0,\\
&D_{4}(L_{n})=0,& &D_{4}(G_n)=L_{n+1},&D_{4}(1)=0.
\end{align*}
\end{prop}
\begin{proof}
Using Lemma \ref{P4lemma} and           the first equation in \eqref{5derivationS},
 we obtain \[\widetilde{f}\in \langle D_{1}\rangle\oplus\langle D_{2}\rangle\oplus \langle D_{3}\rangle\oplus\langle D_{4}\rangle.\]
 If $\widetilde{f}$ is even, there exist $\lambda_{1}$ and $\lambda_{2}$ satisfying $\widetilde{f}=\lambda_{1}D_{1}+\lambda_{2}D_{2}$.\\
If $\widetilde{f}$ is odd, there exist $\lambda_{3}$ and $\lambda_{4}$ satisfying $\widetilde{f}=\lambda_{3}D_{3}+\lambda_{4}D_{4}$.\\
Using \eqref{5derivationS}, $\lambda\equiv0$ and $\widehat{f}\equiv0$, we obtain
\begin{align*}
&\Big(\big[\varphi,\widetilde{f}+\widehat{v}\big]+\big[\delta,v\big]\Big)(x,y)=0\\
&\Rightarrow \varphi(\widetilde{f}(x),\alpha(y))-(-1)^{|x||y|}\varphi(\widetilde{f}(y),\alpha(x))-\widehat{v}(\varphi(x,y))-v(\delta(x,y))=0.\\
&\Rightarrow \left\{ \begin{array}{ll}
\varphi(\widetilde{f}(L_n),\alpha(L_k))-\varphi(\widetilde{f}(L_k),\alpha(L_n))-\widehat{v}(\varphi(L_n,L_k))-v(\delta(L_n,L_k))=0,        \\
\varphi(\widetilde{f}(L_k),\alpha(G_n))-\varphi(\widetilde{f}(G_n),\alpha(L_k))-\widehat{v}(\varphi(L_k,G_n))-v(\delta(L_k,G_n))=0.
\end{array} \right.
\end{align*}
With $\varphi\in\{\varphi_0,\varphi_1\}$, we have :
\begin{align*}
&\varphi(L_1,x)=0,& \forall x\in \mathcal{W}^{q},\\
&\varphi(L_{n} ,L_k)=0,& \forall n+k\neq 0,\\
&\varphi(L_{n} ,G_k)=0,&\forall n+k\neq-1.
\end{align*}

Since $\widetilde{f}\in \{\lambda_{1}D_{1}+\lambda_{2}D_{2},\lambda_{3}D_{3}+\lambda_{4}D_{4}\}$, we obtain
\begin{equation*}
\left\{ \begin{array}{ll}
v(\delta(L_n,L_k))=0 ,       \\
v(\delta(L_k,G_n))=0.\\
\end{array} \right.
\end{equation*}
Thus, we can deduce
 $v\equiv0$ and $\widehat{v}\equiv0$
\end{proof}
\subsection{$2$-cocycles   of Extensions of   Hom-Lie superalgebras}
Let $f\in \mathcal{C}^{2}(K,K)$,  we set $f=\widetilde{f}+\widehat{f}+\overline{f}+v+\widehat{v}+\overline{v}$ where $\widetilde{f}\in \mathcal{C}^{2}(\mathcal{G},\mathcal{G})$,  $\widehat{f}\in \mathcal{C}^{1,1}(\mathcal{G}V,\mathcal{G})$, $\overline{f}
\in \mathcal{C}^{2}(V,\mathcal{G})$
 $v\in \mathcal{C}^{2}(\mathcal{G},V)$, $\widehat{v}\in \mathcal{C}^{1,1}(\mathcal{G}V,V)$ and $\overline{v} \in \mathcal{C}^{2}(V,V)$. 
In this case, we have $[d,f]=\delta_{K}$.
 For all $x,y, z\in\mathcal{G},\ u,v,w \in V$,  we have
\begin{align*}
&[d,f](x,y,z)=\Big([\delta,\widetilde{f}]+[\varphi,\widehat{f}]+[\delta+\lambda,v]+[\varphi,\widetilde{f}+\widehat{v}]\Big)(x,y,z)\\
&\Big([\delta,\widetilde{f}]+[\varphi,\widehat{f}]\Big)(x,y,z)\in \mathcal{G}\\
&\Big([\delta+\lambda,v]+[\varphi,\widetilde{f}+\widehat{v}]\Big)(x,y,z)\in V\\
&[d,f](x,y,w)=\Big([\delta,\widehat{f}+\widehat{v}]+[\varphi,\overline{f}]+[\lambda,\widetilde{f}+\widehat{f}+\widehat{v}]
+[\varphi,\widehat{f}+\widehat{v}+\overline{v}]\Big)(x,y,w)\\
&\Big(\big[\delta,\overline{f}\big]+
\Big( [\lambda,\widehat{f}+\overline{f}+\overline{v}]+[\varphi,\overline{f}+\overline{v}]\Big)(x,u,v)=0\\
&[d,f](u,v,w)=0.
\end{align*}
Therefore, we have the following result
\begin{thm}\label{f5derivation}
 For all $x,y,z\in\mathcal{G},\ u,\ v,\ w\in V$,  we have
\begin{equation}\label{5derivation}
f\ \in Z^2(K,K)   \Leftrightarrow
\left\{ \begin{array}{ll}
[\delta,\widetilde{f}]+[\varphi,\widehat{f}](x,y,z)=0,\\
\big[\delta+\lambda,v\big]+\big[\varphi,\widetilde{f}+\widehat{v}\big](x,y,z)=0,\\
\Big([\delta+\lambda,\widehat{f}]+\big[\varphi,\overline{f}\big]\Big)(x,y,v)=0,\\
\big[  \delta+\lambda,\widehat{v}\big]+[\lambda,\widetilde{f}]+[\varphi,\widehat{f}+\overline{v}\big](x,y,v)=0,\\
\big[\delta,\overline{f}\big](x,u,v)=0,\\
\Big( [\lambda,\widehat{f}+\overline{f}+\overline{v}]+[\varphi,\overline{f}+\overline{v}]\Big)(x,u,v)=0,\\
\big[\lambda,\overline{f}\big](u,v,w)=0.\\
\end{array} \right.
\end{equation}
\end{thm}
\begin{cor}

\begin{enumerate}
\item If $f=\widetilde{f}$, then
\begin{equation}
f\ \in Z^2(K,K)   \Leftrightarrow
\left\{ \begin{array}{ll}
[\delta,\widetilde{f}](x,y,z)=0,\\
\big[\varphi,\widetilde{f}\big](x,y,z)=0,\\
\big[\lambda,\widetilde{f}\big](x,y,v)=0.\\
\end{array} \right.
\end{equation}
\item
If $f=v$, then
\begin{equation}\label{I5derivation}
(f\ \in Z^2(K,K))   \Leftrightarrow (f\ \in Z^2(\mathcal{G},V))\\
\end{equation}
\end{enumerate}
\end{cor}
\bigskip

Now, we assume  that $\mathcal{G}$ is $\mathbb{Z}$-graded and $deg(\varphi)=0$. Let $s=deg(\widetilde{f})=deg(\widehat{f})$.
 If $x\in \mathcal{G}_n$, we set  $x_n=x$.\\
In the first equation in \eqref{5derivation}, we have :
\begin{multicols}{2}
\begin{itemize}
 \item $deg([\delta,\widetilde{f}](x_n,x_m,x_p))=n+m+p+s,$
 \item $deg(\widehat{f}(\alpha(x_n),\varphi(x_m,x_p)))=n+s,$
  \item $deg(\widehat{f}(\alpha(x_m),\varphi(x_p,x_n)))=m+s,$
 \item $deg(\widehat{f}(\alpha(x_p),\varphi(x_n,x_m)))=p+s.$
\end{itemize}
\end{multicols}

In the third equation in \eqref{5derivation}, we have :
\begin{multicols}{2}
\begin{itemize}
 \item $deg([\delta,\widehat{f}])=n+m+s,$
  \item $deg(\widehat{f}(\alpha(x_m),\lambda(x_n,v)))=m+s,$
    \item $deg(\widehat{f}(\alpha(x_m),\lambda(x_n,v)))=n+s,$
     \item $deg(\bar{f} (\beta (v),\varphi (x_n,x_m)))=s$.
\end{itemize}
\end{multicols}
The other terms are of degree zero.
Then, we get
\begin{thm}\label{6derivation}
	If $f\ \in Z^2(K,K) $ and  $n\neq p$, $n\neq m$, $p\neq m$, we have : 

\begin{align}
&[\delta,\widetilde{f}](x_n,x_m,x_p)=0; \forall n+m\neq 0, n+p\neq 0, m+p\neq 0,\label{Vendredi}\\
&\big[\delta,\widetilde{f}](x_n,x_{-n},x_p)+ \widehat{f}(\alpha(x_p),\varphi(x_n,x_{-n}))=0,\label{Vendredi1}\\
 &\widehat{f}(\alpha(x_n),\varphi(x_m,x_p))=0, m+p\neq0,m\neq0,\\
& \big[\delta,\widehat{f}\big](x_n,x_m,u)=0\ \forall n+m\neq0,n\neq 0,  m\neq 0\label{6marsD},\\
 &\widehat{f}(\alpha(x_n),\lambda(x_m,u))=0   \ \forall n\neq m, n \neq 0, m\neq 0,\\
 & \Big[\delta,\widehat{f}](x_0,x_m,u)+\widehat{f}(\alpha(x_m),\lambda(u,x_0))=0, \forall m\neq0\label{6marsDi},\\
  &\overline{f}(\beta(u),\varphi(x_n,x_m))=0, \forall n+m\neq0,n\neq0, m\neq 0,\\
 & \delta(\alpha(x_n),\overline{f}(u,v))=0,\forall n\neq0.
\end{align}
\end{thm}

\subsection{Second cohomology  of   Ramond Hom-superalgebra}
 Let $f$ be an even $2$-cocycle of degree $s$. We can assume that
\begin{align*}
  & \tilde{f}(L_n,L_p)=a_{s,n,p}L_{s+n+p},\  \tilde{f} (L_n,G_p)=b_{s,n,p}G_{s+n+p},\  \tilde{f} (G_n,G_p)=c_{s,n,p}L_{s+n+p},\\
  &\hat{f} (1,L_p)=a_{s,p}'L_{s+p}  \textrm{ and } \hat{f} (1,G_p)=b_{s,p}'G_{s+p}.
\end{align*}
Using \eqref{6marsD}, $\lambda \equiv 0$ and \eqref{6marsDi}, we obtain the following equation
\begin{align}
(\{m\}-\{n\})a_{s,n+m}'
&=(\{m\}-\{n+s\})a_{s,n}'+(\{m+s\}-\{n\})a_{s,m}'.\label{hamza}
\end{align}
This equation was  solved in \cite{Saadaoui}. The solutions are given by $a_{s,n}'=0$ if $s\neq 0$ and 
$a_{0,n}'=na_{0,1}'$ if $s=0$.\\
As above, we obtain 
$b_{s,n}'=0$ if $s\neq 0$.
$b_{0,n}'=b_{0,0}'+na_{0,1}'$.\\
\begin{lem}\label{vendredi3}
 If $f$ is a  $2$-cocycle, we have 
 $$ \bar{v}\equiv0,\quad \hat{f} \equiv0, \quad \bar{f} \equiv0,\quad \text{ and } \quad\widehat{v}\equiv0.$$
\end{lem}
\begin{proof}
Since $ \C$ is  continued in  even part of $HR$, $\forall z,\ z'\in \C$; we have 
 \[\bar{f}(z,z')=zz'\bar{f}(1,1)=0 \qquad  \text{ and }\qquad  \bar{v}(z,z')=zz'\bar{v}(1,1)=0.\]
  
Let $g=\hat{f}(1,.)$. If we assume $[d,f](x_n,x_m,1)=0$, we can deduce 
 \begin{eqnarray*}
  -[\alpha (x_n),g(x_m)]+[\alpha (x_m),g(x_n)]+g([x_n,x_m])-\varphi (\alpha (x_n),g(x_m))\\ +\varphi (\alpha (x_m),g(x_n))+\hat{v} ((1,[x_n,x_m])=0.
 \end{eqnarray*}

Since 
$[\alpha (x_n),g(x_m)]+[\alpha (x_m),g(x_n)]+g([x_n,x_m])\in \mathcal{W}^q$
 and\\ $ -\varphi (\alpha (x_n),g(x_m))+\varphi (\alpha (x_m),g(x_n))+\hat{v} ((1,[x_n,x_m])\in\C$, 
 we obtain  $$-[\alpha (x_n),g(x_m)]+[\alpha (x_m),g(x_n)]+g([x_n,x_m])=0.$$
Then $g$ is an $\alpha$-derivation.

 Recall that the set of $\alpha$-derivations is trivial (see\cite{Saadaoui}). Therefore
$g\equiv 0$ and $\hat{v}(1,.)\equiv 0$.  

\end{proof}

\begin{thm}
$H^2(HR,HR)=\C[\varphi_1] $.
\end{thm}
\begin{proof}
We have  $H^2(\mathcal{W}^{q},\mathcal{W}^{q})=\{0\}$ (see \cite{SaadaouiN}). Then\\
\begin{equation*}
\Big( \delta^2(\widetilde{f})\equiv0\Big)\Rightarrow \Big(\exists \widetilde{g}\in C^1_{\alpha,\alpha}(\mathcal{W}^{q},\mathcal{W}^{q}); \widetilde{f}=\delta^1(\widetilde{g})\Big).
\end{equation*}
Using \eqref{Vendredi}, \eqref{Vendredi1} and Lemma \ref{vendredi3}, we obtain $[\delta ,\tilde{f}]=0$. Then $ \tilde{f}\in Z^2(\mathcal{W}^{q},\mathcal{W}^{q})$.

We have  $\widetilde{f}=\delta^1(\widetilde{g})$, $\hat{f} \equiv0,$ $\bar{f} \equiv0, \quad \bar{v} \equiv0 \quad\text{ and } \quad\widehat{v}\equiv0.$ We deduce, $f=\delta^1(\widetilde{g})+v.$
Therefore, $f=\delta^1_{HR}(\widetilde{g})+w$ where $w(\mathcal{W}^{q},\mathcal{W}^{q})\subset\C$.\\
So
\begin{equation*}
\Big(f\in Z^2(HR,HR)\Big)\Rightarrow\Big( \delta_{HR}^2(w)\equiv0\Big)\Rightarrow \Big(\delta_{T}^2(w)\equiv0 )\Big).
\end{equation*}
As, $H^2(\mathcal{W}^{q},\C)=\C[\varphi_{0}]\oplus \C[\varphi_{1}]$ and $\varphi_0\in B^2_{HR}(\mathcal{W}^{q},\C)$, 
we deduce $w\in \C\varphi_1$.
\end{proof}


\subsection{Second cohomology  of  Special Ramond Hom-superalgebra}
\begin{lem}\label{vendredi4}
 If $f$ is a  $2$-cocycle, we have 
 \[\qquad \bar{v}\equiv0,\qquad \hat{f} \equiv0, \qquad \bar{f} \equiv0,\qquad \text{ and } \qquad\widehat{v}\equiv0.\]
\end{lem}
\begin{proof}
let $g=\hat{f}(1,.)$.
 \begin{align*}
&[d,f](x_n,x_m,1)=0\Rightarrow \\ &-[\alpha (x_n),g(x_m)]+[\alpha (x_m),g(x_n)]+g([x_n,x_m])+\bar{f}(1,\varphi (x_n,x_m))\\
&-\varphi (\alpha (x_n),g(x_m))+\varphi (\alpha (x_m),g(x_n))+\hat{v} (1,[x_n,x_m])+\bar{v}(1,\varphi (x_n,x_m)).
 \end{align*}

Since 
$ -[\alpha (x_n),g(x_m)]+[\alpha (x_m),g(x_n)]+g([x_n,x_m])+\bar{f}(1,\varphi (x_n,x_m))\in \mathcal{W}^q$\\ and $-\varphi (\alpha (x_n),g(x_m))+\varphi (\alpha (x_m),g(x_n))+\hat{v} ((1,[x_n,x_m])+\bar{v}(1,\varphi (x_n,x_m))\in\C$, 
we deduce 
\begin{align}
&-[\alpha (x_n),g(x_m)]+[\alpha (x_m),g(x_n)]+g([x_n,x_m])+\bar{f}(1,\varphi (x_n,x_m))=0.\label{dimanche}\\
 &-\varphi (\alpha (x_n),g(x_m))+\varphi (\alpha (x_m),g(x_n))+\hat{v} ((1,[x_n,x_m])+\bar{v}(1,\varphi (x_n,x_m))=0.\label{dimanche1}\\\nonumber
\end{align}
In \eqref{dimanche}, the term  $\bar{f}(1,\varphi (x_n,x_m))$ is of degree $s$. The other terms are of degree $n+m+s$. Then if $n+m\neq0$, we deduce 
\[-[\alpha (x_n),g(x_m)]+[\alpha (x_m),g(x_n)]+g([x_n,x_m])=0,\]
 If $n+m=0$,  we have $\varphi (x_n,x_m))=0$. Then \[-[\alpha (x_n),g(x_m)]+[\alpha (x_m),g(x_n)]+g([x_n,x_m])=0,\]
which implies that $g$ is an $\alpha$-derivation.
 Recall that the set of $\alpha$-derivations is trivial (see\cite{Saadaoui}).
We deduce $g\equiv 0$ and $\bar{f}(1,.)\equiv 0$.  

Since $g\equiv 0$, the equation \eqref{dimanche1} can be write \[\hat{v} ((1,[x_n,x_m])+\bar{v}(1,\varphi (x_n,x_m))=0.\]
If $n+m=-1$ and $s\neq1$, with $deg(\hat{v} ((1,[x_n,x_m]))=n+m+s$ and  $\hat{v} ((1,[x_n,x_m]))\in\C$,  we obtain $\hat{v} ((1,[x_n,x_m]))=0$.  Thus
 $\bar{v}\equiv0$ and $\hat{v}\equiv 0$.\\
\end{proof}
Then we  obtain the following result about second cohomology of special Ramond Hom-superalgebra.
\begin{thm}
\[H^2(SHR,SHR)=\C[\varphi_0] .\]
\end{thm}
\section{Deformations of    Virasoro Hom-superalgebras}
In this section, we discuss deformations of Ramond Hom-superalgebra and special Ramond Hom-superalgebra.
 \begin{defn}
Let $( \mathcal{G},[.,.]_0,\alpha_0)$ be a  Hom-Lie superalgebra. A one-parameter formal   deformation of $\mathcal{G}$ is given by the $\mathbb{K}[[t]]$-bilinear  and the $\mathbb{K}[[t]]$-linear maps $[.,.]_{t}:\mathcal{G}[[t]]\times\mathcal{G}[[t]]\longrightarrow\mathcal{G}[[t]]$, $\alpha_{t}:\mathcal{G}[[t]]\longrightarrow\mathcal{G}[[t]]$ of the form
\[\displaystyle [.,.]_{t}=\sum_{i\geq 0} t^i[.,.]_{i}\ \text{and}\ \alpha_t=\sum_{i\geq 0}t^i \alpha_{i} ,\]
where each $[.,.]_{i}$ is an even $\mathbb{K}$-bilinear  map $[.,.]_{i}:\mathcal{G}\times\mathcal{G}\longrightarrow\mathcal{G}$ (extended to be $\mathbb{K}[[t]]$-bilinear) and each $\alpha_{i}$ is  an even $\mathbb{K}$-linear map $\alpha_{i}:\mathcal{G}\longrightarrow\mathcal{G}$ (extended to be $\mathbb{K}[[t]]$-linear),
 and satisfying the following conditions
\begin{eqnarray}
&&[x,y]_{t}=-(-1)^{|x||y|}[y,x]_{t},\\
&&\circlearrowleft_{x,y,z}(-1)^{|x||z|}[\alpha_{t}(x),[y,z]_{t}]_{t}=0.\label{Djacobie}
\end{eqnarray}
\end{defn}
\begin{defn}
Let $( \mathcal{G},[.,.]_{0},\alpha_{0})$ be a  Hom-Lie superalgebra. Given two deformations
 $\mathcal{G}_{t}=( \mathcal{G},[.,.]_{t},\alpha_{t})$ and  $\mathcal{G}_{t}'= (\mathcal{G},[.,.]'_{t,},\alpha_{t}')$ of $\mathcal{G}$,  where \[\displaystyle [.,.]_{t}=\sum_{i\geq0}t^i[.,.]_{i},\
[.,.]_{t}'=\sum_{i\geq0}t^i[.,.]_{i}',\ \displaystyle \alpha_{t}=\sum_{i\geq0}t^i\alpha_{i}\textit{ and } \displaystyle \alpha_{t}'=\sum_{i\geq0}t^i\alpha_{i}'.\]
 We say that they
 are equivalent if there exists a formal automorphism
 $\displaystyle \phi_{t}=\sum_{i\geq0} t^i \phi_{i}$,  where $\phi_{i}\in \Big(End (\mathcal{G})\Big)_0$ and 
 $\phi_{0}=id_{\mathcal{G}}$,
 such that
 \begin{eqnarray}\label{gabes12}
  & \phi_{t}([x,y]_{t})=[\phi_{t}(x),\phi_{t}(y)]_{t}',\ \forall x,\ y\in \mathcal{G},\\
   &\text{and } \  \phi_{t}\circ \alpha_t=\alpha'_t\circ  \phi_{t}.
  \end{eqnarray}
   A deformation $\mathcal{G}_{t}$ is said to be trivial if and only if  $\mathcal{G}_{t}$ is equivalent to $\mathcal{G}$ (viewed as a superalgebra on $\mathcal{G}[[t]]$).
    \end{defn}
  
\begin{lem}\label{M} 
Every deformation $ HR_{t }$ of Ramond Hom-superalgebra  such that $\displaystyle [.,.]'_{t}=[.,.]_{0}+\sum_{k\geq p} t^i[.,.]'_{k} $ and $\displaystyle  \alpha_{t}=\displaystyle( \sum_{k\geq 0} a_kt^k)\alpha_0$ is equivalent to a deformation $$\displaystyle[.,.]_t=[.,.]_0+(\sum_{k\geq  p}\lambda_kt^k)\varphi_1.$$
\end{lem}
\begin{proof}
Let   $HR_{t}=( HR,[.,.]'_{t},\alpha_{t})$ be  a deformation of Ramond Hom-superalgebra   $( HR,[.,.]_{0},\alpha_{0})$, where \[\displaystyle [.,.]'_{t}=[.,.]_{0}+\sum_{k\geq p}[.,.]'_{k}t^k\qquad \textit{ and }\qquad   \alpha_{t}=\displaystyle( \sum_{k\geq 0}a_kt^k)\alpha_{0}.\]
Condition \eqref{Djacobie} may be written
 \begin{equation}
\circlearrowleft_{x,y,z}(-1)^{|x||z|}\sum_{s\geq0}t^s\Big(\sum_{k=0}^{s}\sum_{i=0}^{s-k}[\alpha_{i}(x),[y,z]'_{k}]'_{s-i-k}\Big)=0.
  \end{equation}
  This equation is equivalent to the following infinite system :
   \begin{equation}\label{2016}
\circlearrowleft_{x,y,z}(-1)^{|x||z|}\sum_{k=0}^{s}\sum_{i=0}^{s-k}[\alpha_{i}(x),[y,z]'_{k}]'_{s-i-k}=0,\ s=0,1,\dots
  \end{equation}
In particular, 
  for $s=p$, $$\circlearrowleft_{x,y,z}(-1)^{|x||z|}[\alpha_{0}(x),[y,z]_{0}]'_{p}+\circlearrowleft_{x,y,z}(-1)^{|x||z|}[\alpha_{0}(x),[y,z]_{p}']_{0}=0.$$
  Therefore $[.,.]'_p\in Z^2(HR,HR)$. 
  Since $H^2(HR,HR)=\C[\varphi_1]$, 
we deduce  $$[.,.]'_p=-\delta^1(\Phi)+\lambda_{p}\varphi_1,\ \text{ where } \ \Phi\in C^1_{\gamma ,\gamma }, \text{ and } \lambda\in\K .  $$
Let $\Phi_{t}=\Phi_{0}+\Phi t^p$, then $\Phi_{t}^{-1}=\Phi_{0}+\displaystyle \sum_{k\geq 1} (-1)^kt^{kp}\Phi^k$.\\
$[x,y]_{t}=\Phi_{t}^{-1}\Big([\Phi_{t}(x),\Phi_{t}(y)]'_{t}\Big)$.\\
By a simple identification, it follows that $[x,y]_{p}=\lambda_p \varphi_1$. \\
As well $$\displaystyle[.,.]_t=[.,.]_0+t^p\lambda_{p}\varphi_1+\sum_{k> p}t^k[.,.]_k.$$
Thus, by induction we show that
$   [.,.]_{k}=\lambda_{k}\varphi_{1}\ \quad \forall k\geq p.$
\end{proof}
\begin{thm}\label{M2}
Every deformation $ HR_{t }$ such that $\displaystyle [.,.]'_{t}=[.,.]_{0}+\sum_{k\geq p} t^i[.,.]'_{k} $ and $\displaystyle  \alpha_{t}=\displaystyle(\sum_{k\geq 0} a_kt^k)\alpha_0 $ of Ramond Hom-superalgebra is equivalent to a deformation  of the form $$\displaystyle[.,.]_t=[.,.]+\varphi_0+ t \varphi_1.$$
\end{thm}
\begin{proof}
It follows from 
\begin{align*}
&[L_n,G_m]_t=[L_n,G_m]+\sum_{i\geq p}a_i\varphi_1(L_n,G_m)t^i;\\
&[L_n,L_m]_t=[L_n,L_m]+\varphi_0(L_n,L_m);\\
&[L_n,G_m]_t'=[L_n,G_m]+\lambda\varphi_1(L_n,G_m)t;\\
&[L_n,L_m]_t'=[L_n,L_m]+\varphi_0(L_n,L_m).
\end{align*}
$\Phi_0=id;$
$\Phi_s(1)=0,\ \forall s>0;$
$\Phi_s(G_m)=\frac{a_{s+1}}{\lambda}G_m,\quad \forall s\geq p;$
$\Phi_s(L_n)=0,\quad \forall s>0.$
Thus $\displaystyle[.,.]_t=[.,.]+\varphi_0+\lambda t \varphi_1.$ By rescaling we obtain the desired result.
\end{proof}
\begin{thm}\label{M3} 
Every deformation $ SHR_{t }$ such that  $\displaystyle[.,.]'_{t}=[.,.]_{0}+\sum_{k\geq p}t^k[.,.]'_{k}$ and $\displaystyle  \alpha_{t}=\displaystyle  (\sum_{k\geq 0}a_kt^k)\alpha_0$ of special Ramond Hom-superalgebra is equivalent to a deformation
 $$\displaystyle[.,.]_{t}=[.,.]+\varphi_1+ t \varphi_0.$$
\end{thm}
\begin{rem}
Theorem \ref{M3}  can be proved in  the same way as Theorem  \ref{M2}.
\end{rem}

\end{document}